\documentclass[10pt]{amsart}
\usepackage{amsmath}
\usepackage{amsthm}
\usepackage{tikz}
\usepackage{tikz-cd}
\usepackage{amssymb}
\usepackage{graphicx}
\usepackage{thmtools}
\usepackage{thm-restate}
\usepackage{setspace}
\usepackage{gensymb}
\usepackage{cleveref}

\newtheorem{theorem}{Theorem}[section]
\newtheorem{lemma}[theorem]{Lemma}

\theoremstyle{definition}
\newtheorem{definition}[theorem]{Definition}
\newtheorem{example}[theorem]{Example}

\theoremstyle{remark}

\allowdisplaybreaks
\usepackage{comment}
\usepackage{xcolor}
\usepackage[disable]{todonotes}
\begin{document}
\title{Slow escape in tracts}
\author{James Waterman}
\address{School of Mathematics and Statistics, The Open University, Walton Hall, Milton Keynes MK7 6AA, UK}
\curraddr{}
\email{james.waterman@open.ac.uk}
\date{}
\maketitle \begin{abstract}
Let $f$ be a transcendental entire function. By a result of Rippon and Stallard, there exist points whose orbit escapes arbitrarily slowly. By using a range of techniques to prove new covering results, we extend their theorem to prove the existence of points which escape arbitrarily slowly within logarithmic tracts and tracts with certain boundary properties. We then give examples to illustrate our results in a variety of tracts. 
\end{abstract}
\section{Introduction}
Let $f:\mathbb{C} \rightarrow \mathbb{C}$ be a transcendental entire function and denote the $n$th iterate of $f$ by $f^n$, for $n=0,1,2,\ldots$. The \textit{Fatou set}, $F(f)$, is defined to be the subset of $\mathbb{C}$ where the iterates $( f^n)$ of $f$ form a normal family, and its complement is the \textit{Julia set} $J(f)$.  An introduction to properties of these sets can be found in a survey article of Bergweiler \cite{B93}.

The \textit{escaping set} $I(f)$ is
\[I(f)=\{z: f^n(z) \rightarrow \infty ~\text{ as }~ n \rightarrow \infty\},\]
and it is a result of Eremenko in \cite{EremenkoEscape} that $\partial I(f)=J(f)$, $I(f) \cap J(f) \neq \emptyset$, and $\overline{I(f)}$ has no bounded components.
Further, Eremenko conjectured that all components of $I(f)$ are unbounded. While this conjecture is still open, progress has been made by showing that $I(f)$ has at least one unbounded component \cite{AfUnbounded}. This was shown by considering the \textit{fast escaping set} $A(f)$.
This set was first introduced by Bergweiler and Hinkkanen in \cite{OFast} and can be defined by
\[A(f) = \{ z: ~\text{there exists}~ L \in \mathbb{N} ~\text{such that}~ |f^{n+L}(z)|\geq M^n (r) ~\text{for}~ n \in \mathbb{N}\}\] 
as in \cite{Fast}, where 
\[M(r)= \max _{|z|=r}|f(z)| ~ ~ \text{for} ~r>0\]
and $r$ is such that $M(r) >r$ for $r\geq R$. Note $M^n(r)$ denotes iteration of $M(r)$ and $\mathbb{N}$  denotes the non-negative integers.

The set $A(f)$ is independent of $R$ and has similar properties to $I(f)$ in that $\partial A(f) = J(f) $ and $A(f) \cap J(f) \neq \emptyset$. Importantly, all components of $A(f)$ are unbounded by a result of Rippon and Stallard \cite{AfUnbounded}. Further, Bergweiler, Rippon, and Stallard \cite{Tracts} considered fast escaping points in a tract $D$ and defined
\[A(f,D, \rho)=\{z \in D: f^n(z) \in D ~\text{and}~ |f^n(z)| \geq M^n_D(\rho) ~\text{for} ~ n \in \mathbb{N}\}, \]
where
\[M_D(\rho)=\max_{|z|=\rho, ~z\in D} |f(z)|.\] They proved that $A(f,D,\rho)\neq \emptyset$ and that all components of $A(f,D,\rho)$ are unbounded. (The definition of a tract is given at the beginning of Section 2.)

In \cite[Theorem 1]{Slow}, Rippon and Stallard showed there are always escaping points that are not fast escaping. In fact, they showed that there exist points in $J(f)$ that escape arbitrarily slowly. They further proved a two-sided slow escape result \cite[Theorem 2]{Slow}, showing that, for many entire functions, the orbit of a slow escaping point can be controlled to lie between two constant multiples of a specified sequence.

The aim of this paper is to generalize these two results to prove the existence of points which escape arbitrarily slowly within certain tracts. We begin by proving the existence of points which escape arbitrarily slowly within any prescribed logarithmic tract. If there is only one tract, then this result follows directly from Rippon and Stallard's theorem. However, if there is more than one tract, then all we know from their theorem is that there is a point which escapes suitably slowly within the union of these tracts. 
\begin{theorem}\label{MySlow1}
Let $f$ be a transcendental entire function with a logarithmic tract~$D$. Then, given any positive sequence $(a_n)$ such that $a_n \rightarrow \infty$ as $n \rightarrow \infty$, there exists
\[\zeta \in I(f) \cap J(f) \cap \overline{D} ~and~ N \in \mathbb{N}\]
such that
\[f^n ( \zeta) \in \overline{D}, ~\text{for}~ n\geq 1,\]
and
\[|f^n(\zeta)| \leq a_n, ~\text{for}~ n \geq N.\]
\end{theorem}
We also prove a two-sided slow escape result for logarithmic tracts. Note that this result is stronger than Rippon and Stallard's two-sided slow escape result in that the modulus of the iterates can be controlled to both lie between a given sequence and any given constant multiple of that sequence, and further converge to the given sequence.

\begin{theorem}\label{MySlow2}
Let $f$ be a transcendental entire function with a logarithmic tract~$D$. Then, given any positive sequence $(a_n)$ such that $a_n \rightarrow \infty$ as $n \rightarrow \infty$ and satisfying ${a_{n+1}=O(M_D(a_n))}$ as $n \rightarrow \infty$ there exists
\[\zeta \in  J(f) \cap \overline{D} ~and~ N \in \mathbb{N},\]
such that
\[f^n ( \zeta) \in \overline{D}, ~\text{for}~ n\geq 1,\]
and
\[1\leq\frac{|f^n(\zeta)|}{a_n} \leq 1+o(1), ~\text{for}~ n \geq N.\]
\end{theorem}
Note that both \Cref{MySlow1} and \Cref{MySlow2} readily generalize to a prescribed orbit through a finite number of tracts. More care must be taken with an infinite number of tracts. However, with a few minor restrictions on the growth of the given sequence compared to the location of the tracts, these two results still follow from the proofs of \Cref{MySlow1} and \Cref{MySlow2}.

The proofs in \cite{Slow} rely on certain annulus covering results. However, the techniques used to prove these results do not readily generalize to the case of a tract. The proofs of \Cref{MySlow1} and \Cref{MySlow2} instead rely on an annulus covering result obtained by using a derivative estimate  within a logarithmic tract due to Eremenko and Lyubich  \cite{EL92}. In the case where a tract is not logarithmic, this derivative estimate need not hold and in \Cref{HarmSec} we show that if the boundary of the tract is suitably well behaved in a certain precise sense, then we can prove a covering result by using the harmonic measure of a component of the boundary. We say that such tracts have `bounded geometry with respect to harmonic measure' and the covering result enables us to prove the following. 

\begin{theorem}\label{MySlow3}
Let $f$ be a transcendental entire function and let $D$ be a tract of $f$ with bounded geometry with respect to harmonic measure. Then, given any positive sequence $(a_n)$ such that $a_n \rightarrow \infty$ as $n \rightarrow \infty$, there exists
\[\zeta \in I(f) \cap J(f) \cap \overline{D} ~and~ N \in \mathbb{N}\]
such that
\[f^n ( \zeta) \in \overline{D}, ~\text{for}~ n\geq 1,\]
and
\[|f^n(\zeta)| \leq a_n, ~\text{for}~ n \geq N.\]
\end{theorem}

The organization of this paper is the following. Section 2 is devoted to introducing tracts and giving a preliminary result on constructing slow escaping points in general. Section 3 focuses on proving \Cref{MySlow1} and \Cref{MySlow2}. Section 4 focuses on proving \Cref{MySlow3}. Finally, Section 5 gives two examples of tracts to which our results can be applied.

\section{Tracts and constructing slow escaping points}

We recall the classification of the singularities of the inverse function due to Iversen \cite{Iversen} as well as the definition of a tract, following terminology found in \cite{Tracts}. Let $f$ be an entire function and consider $a \in \widehat{\mathbb{C}}$. For $R>0$, let $U_R$ be a component of $f^{-1}(D(a,R))$ (where $D(a,R)$ is the
open disc centered at $a$ with radius $R$ with respect to the spherical metric) chosen so that $R_1< R_2$ implies that $U_{R_1} \subset U_{R_2}$. Then either $\bigcap_{R} U_R = \{z\}$ for some unique $z \in \mathbb{C}$ or $\bigcap_{R} U_R = \emptyset$.

In the first case, $a=f(z)$ and $a$ is an \textit{ordinary point} if $f'(z) \neq 0$, or $a$ is a \textit{critical value} if $f'(z)=0$. In the second case,  $f$ has a \textit{transcendental singularity} over $a$. The transcendental singularity is called \textit{direct} if $f(z) \neq a$ for all $z \in U_R$, for some $R>0$. Otherwise it is \textit{indirect}. Further, a direct singularity is called \textit{logarithmic} if $f: U_R \rightarrow D(a,R) \setminus \{a\}$ is a universal covering. These $U_R$ are called \textit{tracts} for $f$. Note that the structure of these tracts depends on $R$. We will mainly restrict $f$ to a specific tract, which motivates the following definition.

\begin{definition}
Let $D$ be an unbounded domain in $\mathbb{C}$ whose boundary consists of piecewise smooth curves and suppose that the complement of $D$ is unbounded. Further, let $f$ be a complex valued function whose domain of definition contains the closure $\overline{D}$ of $D$. Then $D$ is called a \textit{direct tract} of $f$ if $f$ is holomorphic in $D$, continuous in $\overline{D}$, and if there exists $R>0$ such that $|f(z)|=R$ for $z \in \partial D$ while $|f(z)| > R$ for $ z \in D$. If, in addition, the restriction $f:D \rightarrow \{z \in \mathbb{C}: |z| > R\}$ is a universal covering, then $D$ is called a \textit{logarithmic tract}.
\end{definition}
For transcendental entire functions, all tracts are direct tracts, and so we will take a tract to mean a direct tract. For simplicity, in proving the main results of this paper, we assume that the tract $D$ is a component of $\{z:|f(z)|>1\}$. The proof of the general case is similar. Note, a tract is logarithmic if it contains no critical points and the restriction of the function to the tract has no finite asymptotic values.

In order to prove our results on slow escaping points, the following topological lemma (see, for example, {\cite[Lemma 1]{Slow}}) is used to obtain an orbit passing through a sequence of specified compact sets.
\begin{lemma}\label{TopLemma}
Let $E_n$, $n\geq 0$, be a sequence of nonempty compact sets in $\mathbb{C}$ and $f: \mathbb{C} \rightarrow \hat{\mathbb{C}}$ be a continuous function such that
\[f(E_n) \supset E_{n+1}, ~\text{for}~ n \geq 0.\]
Then there exists $\zeta$ such that $f^n(\zeta) \in E_n$, for $n\geq 0$.
\end{lemma}
This lemma enables us to give the following general recipe for constructing slow escaping points. Note that we will only apply \Cref{Finish} with $m_j=n_j=j$ for $j \in \mathbb{N}$.
\begin{theorem}\label{Finish} Let $f$ be a transcendental entire function and $D$ be an unbounded domain. Suppose there exists a sequence of distinct bounded open sets $\Sigma_n \subset D$ such that $\min \{ |z| : z \in \overline{\Sigma}_n \} \rightarrow \infty $ as $n \rightarrow \infty$ and, for each $n \in \mathbb{N}$,
\begin{equation}\label{EqCover1}
f(\overline{\Sigma}_n)\supset  \overline{\Sigma}_{n+1}.
\end{equation}
Further suppose that there exist increasing sequences $(n_j)$ and $(m_j)$ such that,
\begin{equation}\label{EqCover2}
f(\overline{\Sigma}_{n_j})\supset \overline{\Sigma}_{m_j},
\end{equation}
with $0 \leq m_j \leq n_j$, $j \in \mathbb{N}$, and $m_j \rightarrow \infty$ as $j \rightarrow \infty$. Then, given any positive sequence $(a_n)$ such that $a_n \rightarrow \infty$ as $n \rightarrow \infty$, there exists
\[\zeta \in I(f) \cap J(f) \cap \overline{D} ~and~ N \in \mathbb{N},\]
such that
\[f^n ( \zeta) \in \overline{D}, ~\text{for}~ n\geq 1,\]
and
\[|f^n(\zeta)| \leq a_n, ~\text{for}~ n \geq N.\]
\end{theorem}
\begin{proof}
We will choose a new sequence of open sets from $(\Sigma_n)$ where we repeat some blocks of the sequence of sets $ \Sigma_n$ suitably often in order to hold up the rate of escape of an orbit that passes through each of these sets in turn. The $j$th block is illustrated in the following figure, where the cycle $ \Sigma_{n_j}, \Sigma_{m_j}, \ldots, \Sigma_{n_j-1}$ is repeated $q_j$ times.

\begin{center} 
\begin{tikzpicture}[scale=0.95] 
\draw (0,0) circle (.75cm);
\draw (2.5,0) circle (.75cm);
\draw (6.5,0) circle (.75cm);
\draw (9,0) circle (.75cm);
\draw[->] (-1.5,0) -- (-1,0);
\draw[->] (1,0) -- (1.5,0);
\draw[->] (3.5,0) -- (4,0);
\draw[->] (5,0) -- (5.5,0);
\draw[->] (7.5,0) -- (8,0);
\draw[->] (10,0) -- (10.5,0);
\node[above] at (1.25,0) {$f$};
\node[above] at (-1.25,0) {$f$};
\node[above] at (3.75,0) {$f$};
\node[above] at (5.25,0) {$f$};
\node[above] at (7.75,0) {$f$};
\node[above] at (4.5,1.95) {$f$};
\node[above] at (10.25,0) {$f$};
\node at (4.5,0) {\dots};
\node at (0,0) {$\Sigma_{m_j}$};
\node at (2.5,0) {$\Sigma_{m_j+1}$};
\node at (9,0) {$\Sigma_{n_j}$};
\node at (6.5,0) {$\Sigma_{n_j-1}$};
\draw[->] (8.5,.75) to[bend right] (.5,.75);
\draw (-1,-.5)--(-1,-1)--(10,-1)--(10,-.5);
\draw (4.5,-1)--(4.5,-1.25);
\node[below] at (4.7,-1.25) {$p_j$ sets};
\label{SlowFigure}
\end{tikzpicture}
\end{center}

Now, consider the sequence $(Q_j)$ with $Q_0=0$ and $Q_j=q_1p_1+\ldots + q_jp_j$, for $j \geq 1$. Here $p_j=n_j-m_j+1$ is the length of the $j$th repeated block we will use and $q_j$, the number of repeats of the $j$th block, will be chosen to give a desired rate of escape.
Next, define
\begin{equation}\label{EqCover3}
E_k=
  \begin{cases}
  \overline{\Sigma}_k & \text{for}~ 0 \leq k \leq n_0,\\
\overline{\Sigma}_{k-Q_{j-1}}& \text{for}~ n_{j-1}+Q_{j-1} \leq k \leq n_j +Q_{j-1},~j\geq 1,\\[2.5ex]

         \overline{\Sigma}_{m_j+i} &    \begin{aligned}
       & \text{for}~n_j + Q_{j-1} < k < n_j +Q_j,\\
       &  \text{and}~ i\equiv k-(n_j + Q_{j-1}+1)\;(\bmod\; {p_j}), ~ 0 \leq i < p_j.
    \end{aligned}
  \end{cases}
\end{equation}
Therefore, by (\ref{EqCover1}) and (\ref{EqCover2}), $f(E_k) \supset E_{k+1}$ for $k \geq 0$. So, by \Cref{TopLemma} there exists ${\zeta} \in E_0$ such that 
\begin{equation}\label{EqCover4}
{f^k(\zeta) \in E_k}, ~\text{for}~ k \in \mathbb{N},
\end{equation}
 and hence
\[\zeta \in I(f) \cap \overline{D} ~\text{and}~ f^n(\zeta) \in \overline {D} ~\text{for} ~ n \in \mathbb{N}.\]

Now, note that we can assume $(a_n)$ is an increasing sequence. Choose a sequence $N_{j} \rightarrow \infty$ such that $\max \{ |z| : z \in \overline{\Sigma}_{n}, n \leq n_j \} \leq a_{N_{j}}$, for $j \in \mathbb{N}$. Further, choose $(q_j)$ such that $n_{j-1}+Q_{j-1} \geq N_{j}$ for $j$ sufficiently large. 
Then, by (\ref{EqCover3}) and (\ref{EqCover4}),  \[|f^k(\zeta)| \leq a_{N_{j}}\leq a_{n_{j-1}+Q_{j-1}} \leq a_k,\]  for $n_{j-1} + Q_{j-1} \leq k < n_j + Q_j$,  and $j$ sufficiently large.

Finally, we show that we can ensure that $\zeta \in J(f)$.
Suppose that ${E}_n \subset F(f)$ for some $n \in \mathbb{N}$. Then ${E}_n \subset I(f)$ by normality, since there exists a point $\zeta \in {E}_n \cap I(f)$. However, by using a repeated block and \Cref{TopLemma}, there also exists a point whose orbit remains bounded, which gives  a contradiction. Hence ${E}_n$ meets $J(f)$ for all~$n$. Further, since $J(f)$ is completely invariant,
\[f\left({E}_n \cap J(f)\right) \supset  {E}_{n+1} \cap J(f), ~\text{for}~ n \in \mathbb{N}.\]
Hence, by \Cref{TopLemma} we can choose a point $\zeta \in I(f) \cap J(f) \cap \overline{D}$ for which $f^n(\zeta) \in \overline{D}$, for $n \geq 1$, and  $|f^n(\zeta)| \leq a_n$, for $n$ sufficiently large.
\end{proof}
\section{Slow escape in logarithmic tracts}\label{LogTractsSection}
In this section, we first prove an annulus covering result based on a derivative estimate due to Eremenko and Lyubich \cite{EL92} in a logarithmic tract. This then allows us to prove our result on slow escaping points in a logarithmic tract, \Cref{MySlow1}, and our two-sided slow escape result, \Cref{MySlow2}, by constructing a sequence of compact sets and applying \Cref{Finish}. 

Let $D$ be a logarithmic tract of $f$ and suppose that $f(D)= \mathbb{C}\setminus \overline{\mathbb{D}}$ with $f(0) \in \mathbb{D}$. We consider a logarithmic transform $F$ of $f$ defined by the following commutative diagram,

\[{\begin{tikzcd}
\log D \arrow{r}{F} \arrow[swap]{d}{e^t} & H \arrow{d}{e^t} \\%
z \arrow{r}{f}& w
\end{tikzcd}}
\]
where $\exp(F(t))=f(\exp(t))$ for $t \in \log D$,  $F$ is a conformal isomorphism on each component of $\log D$, and $H= \{z :\operatorname{Re}(z)>0 \}$. Eremenko and Lyubich \cite{EL92} used this logarithmic transform in order to get the following useful expansion estimate.
\begin{lemma}\label{Expansion}
For $z \in D$ as above, we have
\[\left|\frac{ z f'(z)}{f(z)}\right| \geq \frac{1}{4 \pi} \log \left|f(z)\right|.\]
\end{lemma}

We now use this estimate in order to calculate the length of the image of sections of level curves of $f$ and so obtain an annulus covering result. We denote the open annulus by $A(r,R) = \{z : r < |z| < R\}$.
\begin{lemma}\label{BCover}
Let $D$ be a logarithmic tract as above, $c>1$, and $r_0$ be sufficiently large that $M_{D}(r_0)>\exp(\frac{8 \pi^2 c}{c-1})$. If $\Sigma=A(r_0,cr_0) \cap D$, then
\[f(\Sigma) \supset \bar{A}\left(\exp\left(\frac{8 \pi^2 c}{c-1}\right), M_{D}(r_0)\right).\] 
\end{lemma}
\begin{proof}
Consider a level curve, that is a connected component of $|f(z)|=R$, and choose a segment $\sigma= \sigma(R)$ of the level curve such that ${\sigma \subset \Sigma}$, and $\sigma$ meets both $\{z:|z|=r_0\}$ and  $\{z:|z|=c r_0\}$. We will have level curves that fulfill this for all $R \in[1, M_D(r_0)]$. Further, denote by $l(\sigma)$ the length of the curve $\sigma$ and consider the image of $\sigma$ under $f$. Then, by \Cref{Expansion},
\allowdisplaybreaks
\begin{align*}
l(f(\sigma))&= \int_\sigma |f'(z)| \, |dz| \\
& \geq \int_\sigma \frac{1}{4 \pi}{\left|\frac{f(z)}{z}\right| \log|f(z)|} \,|dz| \\
& \geq\int_\sigma \frac{1}{4 \pi}{\frac{R}{c r_0} \log R}\, |dz| \\
&=\frac{1}{4 \pi}{\frac{R}{c r_0} l(\sigma) \log R}  \\
& \geq \frac{1}{4 \pi}{\frac{R}{c r_0} (c r_0-r_0) \log R} .
\end{align*}
Since $f$ has no critical points on $\sigma$, we deduce that $f(\sigma)$ covers a circle of radius $R$ provided that $\frac{1}{4 \pi}{\frac{R}{c r_0} } (c r_0-r_0)\log R \geq 2 \pi R$. This holds if we take  $R \geq \exp(\frac{8 \pi^2 c}{c-1})$. Therefore, \[f(\Sigma) \supset \bar{A}\left(\exp\left(\frac{8 \pi^2 c}{c-1}\right), M_{D}(r_0)\right),\]
as required.
\end{proof}

We are now ready to prove our slow escaping result within a logarithmic tract of a transcendental entire function, \Cref{MySlow1}. We will use a version of \Cref{BCover} with $c=2$ in order to construct a sequence of annuli intersected with our tract and then apply \Cref{Finish} in order to obtain an orbit that escapes suitably slowly.
\begin{proof}[Proof of \Cref{MySlow1}]
Take  $r_0> e^{16 \pi^2}$ sufficiently large that $M_D(r) \geq 4r$ for ${r \geq r_0}$. Then we can apply \Cref{BCover} with $c=2$ to $\Sigma_0=A(r_0,2r_0) \cap D$  to deduce that there exists $r_1 \geq 2 r_0$ such that
\[f(\Sigma_0) \supset  \overline{A}(e^{16 \pi^2},M_D(|r_0|)) \supset \overline{A(r_1,2 r_1) \cap D}= \overline{\Sigma}_1. \]
Further, 
\[ f(\Sigma_0) \supset  \overline{A(r_0, 2 r_0) \cap D}= \overline{\Sigma}_0.\]

Repeating this process we obtain a sequence $r_n \rightarrow \infty$ and a sequence of open sets, $\Sigma_n$, such that
\[f( \overline{\Sigma}_n) \supset \overline{\Sigma}_n \cup \overline{\Sigma}_{n+1}, ~~\text{for $n \geq 0$}.\]
Applying \Cref{Finish}, we obtain the desired result.
\end{proof}

We can also use \Cref{BCover} to prove \Cref{MySlow2}, our two-sided slow escape result. To accomplish this, we use the following consequence of a convexity property of $\log M_D(r)$.
\begin{lemma} \label{M_D^c}
 Let $D$ be a direct tract. Then there exists $R>0$ such that, for all $r>R$ and all $c>1$,
\[M_D(r^c) \geq M_D(r)^c,\]
and thus, for $C>1$,
\[\lim_{r\rightarrow \infty} \frac{M_D(Cr)}{M_D(r)} = \infty.\]
\end{lemma}
The proof is similar to that of \cite[Lemma 2.2]{SmallGrowth} using the convexity of $\log M_D(r)$ with respect to $\log r$ and that $\frac{\log M_D(r)}{\log r} \rightarrow \infty$ as $r \rightarrow \infty$ \cite[Theorem 2.1]{Tracts}.

\begin{proof}[Proof of \Cref{MySlow2}]
Let $C>1$ and let $(a_n)$ be the given positive sequence which satisfies $a_n \rightarrow \infty$ as $n \rightarrow \infty$ and 
\begin{equation}\label{O(Ma_n)}
a_{n+1} \leq K M_D(a_n), ~\text{for} ~ n \geq 0,
\end{equation}
 and for some constant $K>0$. Take $c \in (1,C)$ and choose $N \in \mathbb{N}$ so large that
\begin{equation} \label{N}
a_n> \exp({ \frac{8 \pi^2 c}{c -1}}), ~ M_D(\frac{C}{c} a_n)> Ca_n, ~ \text{and} ~ \frac{M_D(C a_n/c)}{M_D(a_n)} >CK,~ \text{for}~ n \geq N.
\end{equation}
This is possible since $\frac{M_D(r)}{r}\rightarrow \infty$, and $\frac{M_D(Cr/c)}{M_D(r)} \rightarrow \infty$ as $r \rightarrow \infty$ by \Cref{M_D^c}. These conditions will allow us to apply \Cref{BCover}.

 Define ${\Sigma_0=A\left(\frac{C}{c}a_N, C a_N\right) \cap D}$. Note we may assume that $A(a_n, C a_n) \cap D\neq \emptyset$ for $ n \geq N$. Applying \Cref{BCover} with $r_0=\frac{C}{c} a_N$, we obtain, by \eqref{N},
\begin{align}\label{FirstCover}
f(\Sigma_0) \supset  \overline{A}\left(\exp\left({ \frac{8 \pi^2c}{c -1}}\right),M_D\left(\frac{C}{c}a_N\right)\right) \supset \overline{A\left(a_N,C a_N\right) \cap D} \supset \overline{\Sigma}_1, 
\end{align}
where $\Sigma_1=A\left(\frac{C}{c}a_N,C a_N\right) \cap D= \Sigma_0$. Let $\Sigma_n=\Sigma_0$ for $n=2,\ldots, N,$ so that $f(\Sigma_n) \supset {\overline\Sigma}_{n+1}$ for $n < N$, as in \eqref{FirstCover}.
By \Cref{BCover}, \eqref{FirstCover}, and \eqref{O(Ma_n)},
\begin{align*}
 f(\Sigma_N) \supset  \overline{A}\left(\exp\left({ \frac{8 \pi^2c}{c -1}}\right),M_D\left(\frac{C}{c}a_N\right)\right)\supset \overline{A(a_{N+1},C a_{N+1}) \cap D} \supset \overline{\Sigma}_{N+1},
\end{align*}
where $\Sigma_{N+1}=A\left(\frac{C}{c}a_{N+1},C a_{N+1}\right) \cap D$.

We now apply this argument repeatedly for all $n \geq N$ to obtain a sequence of sets $\Sigma_n= A(\frac{C}{c} a_n, C a_n) \cap D$, for $n \geq N$, such that
\begin{equation*}
f(\overline{\Sigma}_n) \supset \overline{A(a_{n+1},C a_{n+1})} \supset \overline{\Sigma}_{n+1}, ~\text{for} ~ n \geq 0.
\end{equation*}
By \Cref{TopLemma} there exists a point $ \zeta \in \overline{\Sigma}_0$ such that
\[f^n(\zeta) \in \overline{\Sigma}_n, ~\text{for}~ n\geq0. \]
Therefore, there exists a point $\zeta \in\overline{D}$ such that  $a_{n} \leq |f^n(\zeta)| \leq Ca_{n}$, for all $n\geq N$.

Next, we show that we can also choose $\zeta \in J(f)$. Since $f$ is bounded on a curve going to $\infty$, $f$ has no unbounded, multiply connected Fatou components (by \cite{Baker75}), so  all the components of $J(f)$ are unbounded (see, for example, \cite[Theorem 1]{Kisaka}). The image of $D$ contains an unbounded connected set in $J(f)$ and so, by complete invariance, $J(f)$ will meet any annulus with sufficiently large radius intersected with $D$. Therefore, $\overline{\Sigma}_n$ meets $J(f)$ for all $n$ sufficiently large and so we can choose $\zeta \in\overline{D} \cap J(f)$ such that  $a_{n} \leq |f^n(\zeta)| \leq Ca_{n}$, for all $n$ sufficiently large.

Finally, we can infer \Cref{MySlow2} by modifying the above proof slightly, choosing $\Sigma_n= A(\frac{C_j}{c_j} a_n, C_j a_n) \cap D$ for ${N_{j-1} < a_n \leq N_j}$, where $C_j \rightarrow 1$ as $N_j \rightarrow \infty$, and $c_j \in (1,C_j), $ for $ j\geq 1$.

\end{proof}

\section{Slow escape in more general tracts}\label{HarmSec}
In \Cref{LogTractsSection} we showed that we can obtain points that escape arbitrarily slowly in a logarithmic tract. We can construct points that escape arbitrarily slowly within a more general direct tract, provided that the boundary of the tract is sufficiently well behaved. First, we prove an annulus covering result based on the hyperbolic metric and then we use this to obtain another annulus covering lemma giving conditions on the harmonic measure and function value. Finally, we apply this covering by estimating some function values of points in the tract compared to the hyperbolic distance between them.

The proof of this first lemma uses the contraction property of the hyperbolic metric. We denote the hyperbolic density at a point $z$ in a domain, or more generally a hyperbolic Riemann surface, $\Sigma$ by $\rho_\Sigma(z)$ and the hyperbolic distance between two points $z_1$ and $z_2$ on $\Sigma$ by $\rho_\Sigma(z_1, z_2)$. This is a more general version of a theorem of Bergweiler, Rippon, and Stallard \cite[Theorem 3.3]{MCWD}.
\begin{lemma}\label{hypcover}
Let $\Sigma$ be a hyperbolic Riemann surface. For a given $K>1$, if $f: \Sigma \rightarrow \mathbb{C} \setminus \{0\}$ is holomorphic, then for all $z_1, z_2 \in \Sigma$ such that
\[\rho_\Sigma (z_1, z_2) < \frac{1}{2} \log \left( 1+ \frac{\log K}{10\pi}\right) ~~~\text{and}~~~ |f(z_2)| \geq K|f(z_1)| \]
we have
\[ f(\Sigma) \supset \bar A(|f(z_1)|,|f(z_2)|).\]
\end{lemma}
\begin{proof}
Suppose that $\rho_\Sigma (z_1, z_2) < \lambda$ and $|f(z_2)| \geq K|f(z_1)| $ for some value of $\lambda$ to be chosen. Suppose also for a contradiction that there exists some point ${w_0 \in  \bar A(|f(z_1)|,|f(z_2)|)\setminus{f(\Sigma)}}$. By Pick's Theorem \cite[Theorem I.4.1]{CarlesonGamelin},
\allowdisplaybreaks
\begin{align*}
\rho_\Sigma (z_1, z_2) & \geq \rho_{f(\Sigma)} (f(z_1),f(z_2)) \\
& \geq \rho_{\mathbb{C} \setminus \{0,w_0\}}(f(z_1),f(z_2)) \\
&=\rho_{\mathbb{C} \setminus \{0,1\}}(f(z_1)/w_0,f(z_2)/w_0).
\end{align*}
Let $\gamma$ be a hyperbolic geodesic in $\mathbb{C} \setminus \{0,1\}$ from $t_1=f(z_1)/w_0$ to $t_2=f(z_2)/w_0$. There exists a segment $\gamma '$ of $\gamma$ joining the point $t_1'$ to $t_2'$, where ${|t_2'|=K|t_1'|}$ and ${1 \in A(|t_1'|, |t_2'|)}$. This choice is possible since $|f(z_2)| \geq K|f(z_1)|$.
Hence we have, ${t_1',t_2' \in \bar A(1/K,K)}$.

From \cite[Theorem 9.13]{Hayman}, the density of the hyperbolic metric on $\mathbb{C}\setminus \{0,1\}$ is bounded below by $1/(2|w|(|\log|w|| + 10\pi))$. Hence,
\allowdisplaybreaks
\begin{align*}
\rho_\Sigma(z_1,z_2) &\geq \rho_{\mathbb{C} \setminus \{0,1\}}(t_1',t_2') \\
&= \int_{\gamma'} \rho_{\mathbb{C} \setminus \{0,1\}}(z) |dz| \\
&\geq \int_{t_1'}^{t_2'} \frac{|dz|}{2|z|(|\log|z|| + 10\pi)} \\
&\geq \int_{|t_1'|}^{|t_2'|} \frac{dr}{2r(|\log r|+10\pi)}\\
&=\int_{|t_1'|}^{1} \frac{dr}{2r(10\pi-\log r)}+\int_{1}^{K|t_1'|} \frac{dr}{2r(10\pi+\log r)}\\
&=-\frac{1}{2} \log\left(1-\frac{\log r}{10\pi}\right)\Big|_{|t_1'|}^{1} +\frac{1}{2} \log\left(1+\frac{\log r}{10\pi}\right)\Big|_{1}^{K|t_1'|}\\
&=\frac{1}{2} \log\left(\left(1-\frac{\log |t_1'|}{10\pi}\right)\left(1+\frac{\log K|t_1'|}{10\pi}\right) \right) \\
&=\frac{1}{2} \log\left( 1 + \frac{\log K}{10\pi}-\frac{\log |t_1'| \log K|t_1'|}{100\pi^2} \right)\\
&\geq \frac{1}{2} \log \left(1+\frac{\log K}{10\pi}\right), ~\mbox{since $1/K \leq |t_1'| \leq 1$. }
\end{align*}
So, if we set $\lambda = \frac{1}{2} \log \left(1+\frac{\log K}{10\pi}\right)$, then
we reach a contradiction to our initial assumption that $\rho_\Sigma (z_1, z_2) < \lambda$.
\end{proof}
Now, we apply \Cref{hypcover} to an annulus or bounded domain intersected with a tract for which we can continue applying \Cref{hypcover} to obtain a slow escape result. This leads us to impose a few extra conditions on the tracts we consider, which we now define.
\begin{definition}\label{BGRHM}
Let $D$ be a direct tract of a function $f$, where $|f(z)|=1$ on $\partial D$, for which there exists a sequence $\Sigma_n$ of quadrilaterals in $D$ tending to $\infty$ each of which contains a point $z_n$ such that  $|f(z_n)|> \max\{|z|: z \in \Sigma_{n+1}\}$ and the harmonic measure in $\Sigma_n$ at $z_n$ of some connected component of $\partial \Sigma_n \cap \partial D$, $\sigma_n$ say, is uniformly bounded from below by some positive value. Then, $D$ is said to have \textit{bounded geometry with respect to harmonic measure}.
\end{definition}
Using the assumptions of \Cref{BGRHM}, consider a general quadrilateral $\Sigma_n$, called $\Sigma$ for simplicity, with its associated point $z \in \Sigma$ and the associated set ${\sigma \subset \partial \Sigma \cap \partial D}$. Then there exists a hyperbolic geodesic $\gamma$ joining $z$ to $\sigma$ so that $\sigma$ is invariant under hyperbolic reflection in $\gamma$.
Consider the Riemann map $\phi:\Sigma \rightarrow \mathbb{D}$ such that $\phi(z)=0$ and $\phi( \gamma)$ is the interval $[0,1)$. Then $1$ is the midpoint of the arc $\phi(\sigma)$. Let $\theta' \in (0,\pi)$ be the infimal angle of $\theta \in (0, \pi)$ for which $|f(\phi^{-1}(e^{i\theta}))| \neq 1$ and take $\eta \in (\cos \theta',1)$. 
\begin{center}
\begin{tikzpicture}[scale=0.95]
\draw [->] (10.5,0) -- (12,0);
\draw  (15,0) circle [radius=2];
\draw  (6,0) arc [radius=2, start angle=0, end angle= 45];
\draw  (6,0) arc [radius=2, start angle=0, end angle= -45];
\draw  (5,1) to [out=20,in=160] (7,1)
to [out=-20,in=160] (7,1) to [out=-20,in=-155] (9.5,1);
\draw  (5,-1) to [out=20,in=160] (7,-1)
to [out=-20,in=160] (7,-1) to [out=-20,in=-155] (9.5,-1);
\draw  (9,0) arc [radius=2, start angle=0, end angle= 45];
\draw  (9,0) arc [radius=2, start angle=0, end angle= -55];
\draw (7,1) to [out=-20,in=160] (7,0);
\draw[fill] (7,0) circle [radius=0.025];
\node[left] at (7,0) {$z$};
\node at (7.2,0.3) {$\gamma$};
\node at (8.2,-0.7) {\LARGE $\Sigma$};
\node at (11.25,0.5) {$\phi$};
\node at (7,1.2) {$\sigma$};
\draw[fill] (15,0) circle [radius=0.025];
\node [below] at (15,0) {$0$};
\draw [ultra thick] (17,0) arc [radius=2, start angle=0, end angle= 45];
\draw [ultra thick] (17,0) arc [radius=2, start angle=0, end angle= -45];
\draw [ultra thick] (13,0) arc [radius=2, start angle=180, end angle= 160];
\draw [ultra thick] (13,0) arc [radius=2, start angle=180, end angle=  200];
\node [right] at (16.9,0.8) {$\phi(\sigma)$};
\node [right] at (16.4,1.6) {$e^{i\theta'}$};
\draw  (15,0) -- (17,0);
\draw[fill] (16.5,0) circle [radius=0.025];
\node [above] at (16.5,0) {$\eta$};
\node [below] at (16,0) {$\phi(\gamma)$};
\node at (14.1,-1.1) {\LARGE $\mathbb{D}$};
\end{tikzpicture}
\end{center}

In the following $P_\theta(\eta)=(1-\eta^2)/|\eta-e^{i\theta}|^2$ denotes the Poisson kernel of $\mathbb{D}$ with singularity at $e^{i \theta}$.
\begin{lemma}\label{MCover}
Let $f$, $\Sigma$, $\sigma$, $\phi$, $\theta'$, and $\eta \in (\cos \theta',1)$  be as above and  such that
\begin{align}\label{Mbound}
\log |f(\phi^{-1}(0))|&>{\frac{20\pi \eta}{(1-P_{\theta'}( \eta))(1-\eta)}}.
\end{align}
Then,
\[ f(\Sigma) \supset \bar A(|f(\phi^{-1}(\eta))|, |f(\phi^{-1}(0))|).\]
\end{lemma}
\begin{proof}
Consider $u(z)= \log|f(\phi^{-1}(z))|$. We estimate the value of the function $u$ at $\eta$.
Since $u$ is harmonic in $\mathbb{D}$ and vanishes on $\phi(\sigma)$,
\begin{align*}
u(\eta) & = \frac{1}{2 \pi} \int_0^{2 \pi}  P_{\theta} (\eta) u(e^{i \theta}) d \theta \\
& \leq \frac{P_{\theta'}( \eta)}{2 \pi} \int_0^{2 \pi}  u(e^{i \theta}) d \theta\\
&= P_{\theta'}( \eta) u(0),
\end{align*}
by the definition of $\theta'$ and the mean value theorem. Hence,
\[u(0)- u(\eta) \geq (1-P_{\theta'}( \eta)) u(0).\]
Now we let $\log K = (1-P_{\theta'}( \eta)) u(0)$. Then $K>1$ since $P_{\theta'}( \eta) <1$.
Therefore, by the formula for the hyperbolic distance in \cite[page 688]{Hayman},
\allowdisplaybreaks
\begin{align*}
\rho_\mathbb{D} (0, \eta) &= \frac{1}{2} \log\left(\frac{1+\eta}{1-\eta}\right) < \frac{1}{2} \log \left(1+ \frac{\log K}{10\pi}\right)\\
&\iff \left( \frac{1+\eta}{1-\eta}\right) <1+ \frac{\log K}{10\pi} \\
&\iff 10\pi \left( \frac{2 \eta}{1-\eta} \right) < \log K \\
& \iff 10\pi \left( \frac{2 \eta}{1-\eta} \right) < (1-P_{\theta'}( \eta)) u(0) \\
&\iff u(0)>{\frac{20\pi \eta}{(1-P_{\theta'}( \eta))(1-\eta)}}.
\end{align*}
Hence, by \eqref{Mbound}, the conditions of \Cref{hypcover} are satisfied and the result follows.
\end{proof}
We now apply the previous lemmas in order to prove \Cref{MySlow3}. We first estimate the function value inside the tract compared to the hyperbolic distance. We then apply \Cref{MCover} and construct a sequence of domains to which we can apply \Cref{Finish}.

\begin{proof}[Proof of \Cref{MySlow3}]
Without loss of generality we can let $a_n$ be any increasing positive sequence such that  $a_n \rightarrow \infty$ as $n \rightarrow \infty$.
Let $D$ be a direct tract of $f$ with bounded geometry with respect to harmonic measure. Consider a quadrilateral $\Sigma_n$ and the point $z_n \in \Sigma_n$ as in \Cref{BGRHM}. Then by \Cref{BGRHM} there exists $\varepsilon$ independent of $n$, $\sigma_n$, and $\theta_n'$ such that $\theta_n' \geq \varepsilon>0$ for all $n \geq 0$. Let $\eta_n \in (\cos \theta_n',1)$ be chosen so that
\begin{equation}\label{Ptheta}
P_{\theta'_n}( \eta_n)\leq \frac{C(\varepsilon)}{\log |f(z_n)|},
\end{equation}
where $C(\varepsilon)$ is a constant, to be chosen, that depends solely on $\varepsilon$.
To apply \Cref{MCover}, we further need to choose $\eta_n$ such that
\begin{equation}\label{MD}
\log |f(z_n)|>{\frac{20\pi \eta_n}{(1-P_{\theta'_n}( \eta_n))(1-\eta_n)}}.
\end{equation}
 So, we are choosing $\eta_n$ very close to $1$ and $P_{\theta'_n}( \eta_n)$ close to $0$ for sufficiently large $|f(z_n)|$. We want \eqref{Ptheta}  and \eqref{MD} not to conflict. 

We have by definition that $P_{\theta'_n}(\eta_n)=(1-\eta_n^2)/{|\eta_n-e^{i\theta'_n}|^2}$, so we want both
\[(1-\eta_n)\log |f(z_n)| \leq \frac{C(\varepsilon)}{1+\eta_n} |\eta_n-e^{i\theta'_n}|^2 \]
and
\[(1-\eta_n)\log |f(z_n)|> \frac{20\pi \eta_n}{1-P_{\theta'_n}(\eta_n)} \approx 20\pi,\]
from \eqref{Ptheta} and \eqref{MD}, to be true. First, since $\cos \theta'_n < \eta_n < 1$, we can observe that
\[| \eta_n - e^{i \theta'_n}| > \sin \theta'_n \geq \sin \varepsilon.\]
So, it is sufficient to choose $\eta_n \in (\cos \theta'_n,1)$ such that

\[40\pi\leq(1-\eta_n)\log |f(z_n)| \leq \frac{1}{2}C(\varepsilon)\sin^2 \varepsilon. \]
This choice is possible if  ${C(\varepsilon)=160 \pi/ \sin^2 \varepsilon}$ and ${\log |f(z_n)| \geq 40 \pi/ (1-\cos \varepsilon)}$, for $n\geq N$ say, (for example, take $\eta_n=1-40\pi/\log|f(z_n)|$).

Hence, by \Cref{MCover} and \Cref{BGRHM},
\[f(\Sigma_n) \supset  \overline{\Sigma}_{n+1}, ~ \text{for}~ n \geq N. \]
Further, we may assume that
\[f(\Sigma_n) \supset  \overline{\Sigma}_n, ~ \text{for}~ n \geq N,\]
since we have $\log|f(\phi^{-1}(\eta_n))|\leq P_{\theta'_n}(\eta_n) \log |f(z_n)| < C(\varepsilon)$, by \eqref{Ptheta} and the reasoning at the start of the proof of \Cref{MCover}.

Relabeling, we obtain a sequence of domains of the form $\Sigma_n$, such that
\[f( \overline{\Sigma}_n) \supset \overline{\Sigma}_n \cup \overline{\Sigma}_{n+1}, ~~\text{for $n \geq 0$}.\]
Applying \Cref{Finish}, we obtain the desired result.
\end{proof}
\section{Examples}
In this section, we give two concrete examples which demonstrate both the kinds of tracts that exist and some methods with which to construct points that escape arbitrarily slowly in these tracts. 
First, we give an example of a function to which we can apply \Cref{hypcover} to construct slow escaping points in a direct tract with no logarithmic singularities. In forthcoming work, we will show this function is in the class $\mathcal{B}$, thus giving an example of a function in the class $\mathcal{B}$ with no logarithmic singularity over any finite value. This example is the reciprocal of the entire function studied in \cite{BE08}; for an illustration of the tracts of this function, see \cite[Figure 1]{BE08}.
\begin{example}
Consider the entire function
\[f(z)=\exp(-g(z)), ~\text{where} ~g(z)=\sum_{k=1}^{\infty} \left(\frac{z}{2^k}\right)^{2^k}.\]
Then, 
\begin{enumerate}
\item no direct tract $D$ of $f$ contains any logarithmic tracts, and
\item there exists $z \in D$ such that $f^n(z) \rightarrow \infty$ arbitrarily slowly to any given direct singularity of $f$.
\end{enumerate}

First, we recall from Bergweiler and Eremenko \cite[Section 6]{BE08} that $\exp(g(z))$ has uncountably many direct singularities over $0$, but no logarithmic singularity over any finite value. So, $f$ has infinitely many direct singularities over $\infty$, but no logarithmic singularity over $\infty$, and hence no tract of $f$ contains any logarithmic tracts. In order to prove that  $\exp(g(z))$ has infinitely many direct singularities over $0$, but no logarithmic singularity over any finite value, they show that there is an infinite binary tree in $\mathbb{C}$ for which every unbounded path on the tree is an asymptotic path along which $\operatorname{Re} g(z) \rightarrow -\infty$ as $z \rightarrow \infty$. The direct singularities correspond to different branches of this infinite tree. We shall use the estimates they give in obtaining this result in order to prove (2).

We now introduce the following notation and results from \cite[Section 6]{BE08}. We fix an $\varepsilon$ with $0<\varepsilon\leq \frac{1}{8}$ and set $r_n=(1+\varepsilon)2^{n+1}$ and $r_n'=(1-2\varepsilon)2^{n+2}$ for $n\in \mathbb{N}$. Then for $j\in\{0,1,\cdots,2^n-1\}$ we define the sets
\[A_{j,n}=\left\{r\exp\left(\frac{2\pi ij}{2^n}\right) : r\geq r_n \right\}\]
and
\[B_{j,n}=\left\{r\exp\left(\frac{\pi i}{2^n}+\frac{2\pi ij}{2^n}\right) :r_n\leq r\leq r'_n \right\}.\]
From \cite{BE08}, we have that $\operatorname{Re} g(z) >2^{2^n}$ for $z\in A_{j,n}$ and $\operatorname{Re} g(z) <-2^{2^n}$ for $z\in B_{j,n}$. Further, $\arg g(r e^{i \theta})$ is an increasing function of $\theta$, for $ r_n \leq r \leq r_n'$, and it increases by $2^n2\pi$ as $\theta$ increases by $2 \pi$.

 We show we have slow escaping points in the tracts of this function and can control the orbits of these points as they escape. To do this we need to further estimate the function in order to apply \Cref{hypcover}.

Fix $j$ and $n$ as above, and consider the annular sector, $\Sigma_{j,n}$, about $B_{j,n}$, bounded on the sides by $A_{j,n}$ to $A_{j+1,n}$ and from $r_n$ to $r_n'$. That is,
\[\Sigma_{j,n}=\left\{r e^{i\theta}: r_n \leq r \leq r_n' ~\text{and}~ \frac{2\pi j}{2^n} \leq  \theta \leq \frac{2\pi (j+1)}{2^n} \right\}.\]
 This domain will meet exactly one connected component of the tract intersected with $\overline{A}(r_n, r_n')$, by a counting argument, since $\arg g(r e^{i \theta})$ is an increasing function of $\theta$, and it increases by $2^n2\pi$ as $\theta$ increases by $2 \pi$. 
Let $z=r\exp\left(\frac{\pi i}{{2^{n+2}}}+\frac{2\pi i j} {2^n}\right)$ for $r_n \leq r \leq r_n'$ and $s=r/2^n \in [2(1+\varepsilon),4(1-2 \varepsilon)]$. Then we can evaluate the following terms of $g(z)$:
\[n\text{th term:}~~ s^{{2^n}}(\cos(\pi/4)+i \sin (\pi/4))= s^{{2^n}}(1/\sqrt{2}+i/\sqrt{2})\]
\[(n+1)\text{st term:}~~ (s/2)^{{2^{n+1}}}(\cos(\pi/2)+i \sin (\pi/2))= (s/2)^{{2^{n+1}}}i\]
\[(n+2)\text{nd term:}~~ (s/4)^{{2^{n+2}}}(-1)=-(s/4)^{{2^{n+2}}}\]
\[(n+3)\text{rd term:}~~ (s/8)^{{2^{n+3}}}.\]
So, we have that, for $z$ as above and where $C$ is some absolute constant,
\begin{align*}
|f(z)| &\leq \exp\left(s^{2^{n-1}}(n-1)2^{2^{n-1}}-s^{2^n}/\sqrt{2}-(s/4)^{{2^{n+2}}}-(s/8)^{{2^{n+3}}}-\ldots\right)\\
&\leq \exp\left(s^{2^{n-1}}\left((n-1)2^{2^{n-1}}-s^{2^{n-1}}/\sqrt{2}\right)-C\right)\\
&\leq \exp\left(s^{2^{n-1}}2^{2^{n-1}}\left((n-1)-(1+\varepsilon)^{2^{n-1}}/\sqrt{2}\right)-C\right) \rightarrow 0 ~ \text{as $n\rightarrow \infty$.}
\end{align*}
We now choose two points in $\Sigma_{j,n}$: \[z_1=\frac{r_n+r_n'}{2}\exp\left(\frac{\pi i}{{2^{n+2}}}+\frac{2\pi i j} {2^n}\right) ~\text{and}~ z_2=\frac{r_n+r_n'}{2}\exp\left(\frac{\pi i}{2^n}+\frac{2\pi ij}{2^n}\right).\] Then, $\rho_{\Sigma_{j,n}}(z_1, z_2)$ is bounded by some absolute constant. Since $z_2\in B_{j,n}$, we deduce that, for $n \geq N$ say, we have
\[\rho_{\Sigma_{j,n}}(z_1, z_2)<\frac{1}{2} \log \left( 1+ \frac{1}{10\pi}\log \frac{|f(z_2)|}{|f(z_1)|}\right).\]
So, by \Cref{hypcover},  $f(\Sigma_{j,n}) \supset \bar A(|f(z_1)|,|f(z_2)|) \supset \Sigma_{j,n} \cup \Sigma_{j,n+1}$ for $n \geq N$ and so we can apply \Cref{Finish} to obtain points that escape as slowly as we wish through any sequence $(\Sigma_{j_n,n})_{n\geq N}$, where $j_n \in \{0,1, \ldots, 2^n-1\}$, and in particular to any direct singularity.
\end{example}
Next, we give an example to show how we can sometimes obtain slow escaping points in a tract that is not simply connected.  We use \Cref{MySlow3} and then give an alternative method involving the zeros of the function. Note that the latter method relies on the existence of large critical values.
\begin{example} Consider the function 
\[f(z)=e^{z^2}\cos z .\]
Then $f$ has two multiply connected non-logarithmic tracts, each of which contains orbits of points that escape arbitrarily slowly.

\begin{figure}
\centering
\fbox{\includegraphics[width=7cm, height=7cm]{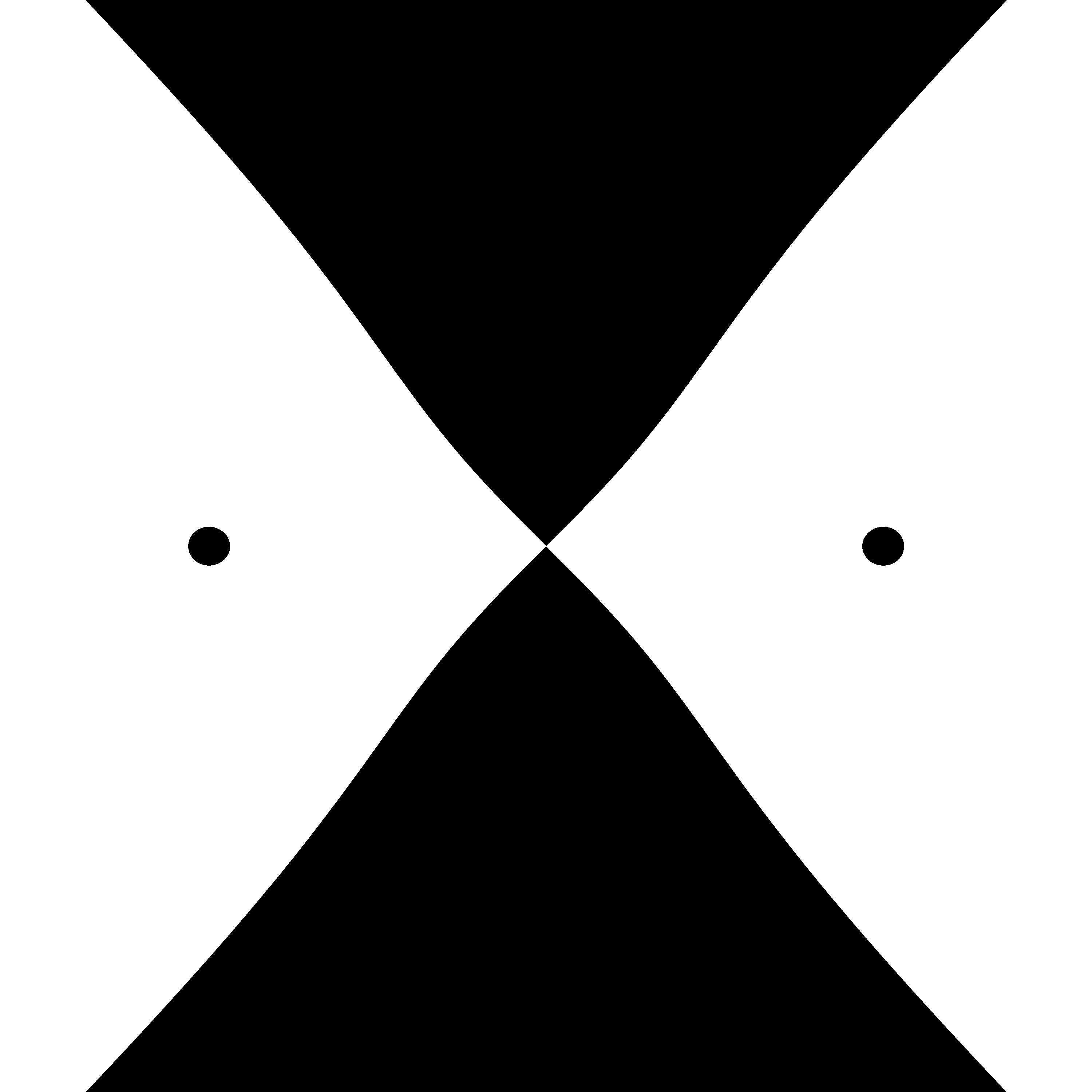}}
\caption{The tracts of $e^{z^2}\cos(z)$ in white, showing holes around the zeros at $\pm \frac{\pi}{2}$.}
\end{figure}

First, $f$ has zeros at $(2n+1) \frac{\pi}{2}$ for $n\in \mathbb{Z}$.  Call the rightmost tract $D$, where $|f(z)|=1$ on $\partial D$. As $e^{z^2}$ has two tracts that are quadrants with one symmetrically containing the positive real axis and the other the negative real axis, and $e^{z^2}$ grows much faster than $\cos|z|$, it is natural to assume the shape of the tract of $f$ will be similar away from the zeros of $\cos(z)$. In fact, it is easy to check that $\partial D$ consists of a Jordan curve passing through $\infty$ and lying in ${\{x+iy: x \geq 0, x \geq |y|\}}$ together with infinitely many bounded Jordan curves surrounding the zeros at $(2n+1) \frac{\pi}{2}$. 

Now, let ${\Sigma_n=A\left((2n+1)\frac{\pi}{2},(2n+3) \frac{\pi}{2}\right) \cap D \cap \mathbb{H}}$, for $n \geq 0$, and where $\mathbb{H}$ denotes the upper half-plane. Each $\Sigma_n$ is a simply connected domain that contains $z_n=(2n+2) \frac{\pi}{2} \exp({\frac{i \pi}{8}})$. This choice of $z_n$ ensures that $|f(z_n)| \geq (2n+5)\frac{\pi}{2}$. It is not difficult to see that the harmonic measure in $\Sigma_n$ at $z_n$ of the largest boundary component of $\Sigma_n \cap \partial D$ is uniformly bounded from below with respect to $n$. This can be seen by viewing $\Sigma_n$ as a quadrilateral with one side on the boundary of the tract and the ratio of the sides uniformly bounded.  Therefore, we can apply \Cref{MySlow3} to obtain slow escaping points in the tract $D$.

Finally, we outline another method to obtain slow escaping points for this function in the tract $D$. The critical points of $f$ satisfy ${z=\frac{1}{2} \tan(z)}$. These occur for large positive $z$ at $z=\frac{(2n+1) \pi}{2}-\varepsilon_n$, where $\varepsilon_n>0$. For such a point $z$, ${|\cos(z)| = \sin(\varepsilon_n)  \approx  \varepsilon_n}$.
We also have \[\frac{(2n+1) \pi}{2}-\varepsilon_n=\frac{1}{2} \tan\left(\frac{(2n+1) \pi}{2}-\varepsilon_n\right)\] and \[\left|\tan\left(\frac{(2n+1) \pi}{2}-\varepsilon_n\right)\right|=\frac{\cos(\varepsilon_n)}{\sin({\varepsilon_n})} \approx \frac{1-\frac{1}{2}\varepsilon_n^2}{\varepsilon_n} \approx {\frac{1}{\varepsilon_n}}.\] Hence, $\varepsilon_n \approx \frac{1}{(2n+1) \pi}$. Therefore the critical values of $f$ are approximately \[{\frac{1}{(2n+1) \pi} \exp\left(\frac{(2n+1)^2 \pi^2}{4} \right)}.\] 
The images of the level curves around a zero form circles for function values up to at least the critical value of the nearest critical point. By estimating $f$ on a quadrilateral containing such a critical point and a zero, we see that the domain $A\left(2n \frac{\pi}{2},(2n+2) \frac{\pi}{2}\right) \cap D$ contains all of these level curves and that the critical value of $f$ in this domain has modulus much larger than $(2n+4) \frac{\pi}{2}$. Hence, the conditions for \Cref{Finish} are satisfied and we again obtain points that escape arbitrarily slowly in $D$.
\end{example}

\subsection*{Acknowledgments}
I would like to thank my supervisors Prof. Phil Rippon and Prof. Gwyneth Stallard for their patient help and guidance in the preparation of this paper. I would also like to thank Vasiliki Evdoridou for many helpful comments.
\bibliographystyle{abbrv}
\bibliography{SlowEscapeArxiv}{}
\end{document}